\documentclass{amsart}
\usepackage[left=1in,right=1in,top=1.2in,bottom=1in]{geometry}
\usepackage[utf8]{inputenc}
\usepackage{amsmath,amssymb}
\usepackage{comment}
\usepackage{xcolor}
\usepackage{url}
\usepackage[colorlinks=true,allcolors=blue,backref=page]{hyperref}
\usepackage[noabbrev,capitalize,nameinlink]{cleveref}

\newtheorem{theorem}{Theorem}

\newtheorem{lemma}[theorem]{Lemma}

\newtheorem{corollary}[theorem]{Corollary}

\newtheorem{fact}[theorem]{Fact}

\newtheorem{prop}[theorem]{Proposition}

\newtheorem*{claim*}{Claim}

\theoremstyle{remark}
\newtheorem*{remark*}{Remark}

\numberwithin{theorem}{section}

\renewcommand{\phi}{\varphi}
\renewcommand{\emptyset}{\varnothing}

\renewcommand{\leq}{\le}
\renewcommand{\geq}{\ge}

\newcommand{\cE}{\mathcal E}

\newcommand{\cM}{\mathcal M}

\newcommand{\E}{\bE}

\newcommand{\one}{{\mathbf{1}}}

\newcommand{\bE}{\mathbb E}

\renewcommand{\le}{\leqslant}
\renewcommand{\ge}{\geqslant}

\newcommand{\R}{\mathbb R}
\newcommand{\C}{\mathbb{C}}

\newcommand{\bN}{\mathbb N}

\newcommand{\Chat}{\widehat{\C}}
\newcommand{\mhat}{\widehat{m}}

\renewcommand{\P}{\mathbb{P}}
\newcommand{\nuhat}{\widehat{\nu}}

\title
[]{Almost sure behavior of the zeros of iterated derivatives of random polynomials}

\author{Marcus Michelen}
\author{Xuan-Truong Vu}

\begin{document}
	\begin{abstract}
 Let $Z_1,\, Z_2,\dots$ be independent and identically distributed complex random variables with common distribution $\mu$ and set 
$$ 
P_n(z) := (z - Z_1)\cdots (z - Z_n)\,.
$$
 Recently, Angst, Malicet and Poly proved that the critical points of $P_n$ converge in an almost-sure sense to the measure $\mu$ as $n$ tends to infinity, thereby confirming a conjecture of Cheung-Ng-Yam and Kabluchko. 
 In this short note, we prove for any fixed $k\in \mathbb{N}$, the empirical measure of zeros of the $k$th derivative of $P_n$ converges to $\mu$ in the almost sure sense, as conjectured by Angst-Malicet-Poly.
	\end{abstract}	
\maketitle

\section{Introduction}

Let $\mu$ be a probability measure on $\C$ and let $(Z_j)_{j\geq 1}$ be independent and identically distributed (i.i.d.) complex random variables with distribution $\mu$.  Define the sequence of random polynomials $(P_n)_{n\geq 1}$ via \begin{equation}
    P_n(z) := (z - Z_1)\cdots (z - Z_n)\,.
\end{equation} 
Pemantle and Rivin \cite{pemantle2013distribution} introduced this model and conjectured that the critical points of $P_n$ are close to the roots of $P_n$.  More rigorously,  for each fixed $k \in \bN$
let $\nu_n^{(k)}$ to be the empirical measure of $P_n^{(k)}$, i.e. 
\begin{equation}
    \nu_n^{(k)} := \frac{1}{n-k} \sum_{z \in \C : P_n^{(k)}(z) = 0} \delta_{z}\,,
\end{equation}
and let $\mu_n$ denote the \emph{empirical measure} of $P_n$, i.e.\ \begin{equation}
    \mu_n := \frac{1}{n}\sum_{z \in \C : P_n(z) = 0} \delta_{z}\,,
\end{equation} 
where $\delta_y$ denotes the point mass at $y$.

Pemantle and Rivin conjectured that in the case of $k = 1$, we have $\nu_n^{(1)}$ converges weakly to $\mu$; Pemantle and Rivin proved this under the assumption that $\mu$ has finite $1$-energy.  Kabluchko \cite{kabluchko2015critical} proved Pemantle and Rivin's conjecture, showing that $\nu_n^{(1)} \to \mu$ in probability as $n \to \infty$.  Kabluchko's result was extended by Byun, Lee and Reddy \cite{byun-lee-reddy} who proved that for each fixed $k \in \bN$ one has that $\nu_n^{(k)}$ converges weakly to $\mu$ in probability.  Recently, the authors showed the same holds if $k$ grows slightly slower than logarithmically in $n$ \cite{michelen2022zeros}.  The works \cite{byun-lee-reddy} and \cite{michelen2022zeros} on convergence of higher derivatives follow the same general strategy as Kabluchko's original proof \cite{kabluchko2015critical}, which much of the new ingredients coming in to handle an anti-concentration estimate.  For more references on this model and adjacent models, see the works \cite{alazard2022dynamics,cheung2014critical,cheung2015higher,hoskins2020dynamics,hoskins2022semicircle,kabluchko2021repeated,kiselev2022flow,o2021nonlocal,steinerberger2020free,subramanian}
and the references therein.

Cheung-Ng-Yam \cite{cheung2014critical} and independently Kabluchko (see \cite{michelen2022zeros}) conjectured that in fact $\nu_n^{(1)}$ should weakly converge to $\mu$ \emph{almost surely} and not just in probability.  This was recently proven by Angst, Malicet and Poly \cite{angst2023almost}  for all probability measures $\mu$.  Angst, Malicet and Poly  also conjectured \cite{angst2023almost}  that the almost-sure convergence of $\nu_n^{(k)}$ should hold for each fixed $k\in\bN$.  In this short note, we confirm their conjecture.

\begin{theorem}\label{thm:main}
Almost surely with respect to $\P$, for each fixed $k \in \bN$ the sequence of empirical measure $\nu_n^{(k)}$ converges to $\mu$ as $n$ tends to infinity.
\end{theorem}

Our proof of \cref{thm:main} takes inspiration from Angst, Malicet and Poly's proof of the $k = 1$ case \cite{angst2023almost}; the new ingredient is to handle a non-linear, high-dimensional, multivariate anti-concentration problem via a decoupling approach.  We first outline the general shape of the Angst-Malicet-Poly approach as well as where our contribution comes into play in \cref{ss:outline}.  We then prove our anti-concentration estimate in \cref{sec:anticon} and complete the proof of \cref{thm:main} in \cref{sec:proof}.

\subsection{Outline of the Angst-Malicet-Poly strategy and our contribution}\label{ss:outline}

The main engine behind Angst-Malicet-Poly is a simple-yet-powerful fact about probability measures on the Riemann sphere.  To set up their Lemma, set $\Chat$ to be the Riemann sphere,  let $\cM := \{ \psi(z) = \frac{az + b}{cz + d} \}$ be the set of M\"obius transformations, and let $\lambda_\cM$ be the measure on $\cM$ inherited by setting taking the complex Lebesgue measure on the tuple $(a,b,c,d)$.  Define $\log_- z = | \log z| \one_{z \in [0,1]}$.   Their main engine is the following lemma \cite[Lemma 2.7]{angst2023almost}:

\begin{lemma} \label{lem:AMP-sphere-equality}
    Let $\mhat_1$ and $\mhat_2$ be two probability measures on $\Chat$ so that \begin{equation}\label{eq:AMP-engine-eq}
    \int_{\Chat} \log_- |\psi(z)| \,d\mhat_1(z) \leq \int_{\Chat} \log_- |\psi(z)| \,d\mhat_2(z) \end{equation}
    for almost-every M\"obius transformation $\psi \in \cM$.  Then $\mhat_1 = \mhat_2$.
\end{lemma}

An appealing aspect of this Lemma is that it requires only a one-sided bound; further, the space of probability measures on $\Chat$ is compact, and so it will suffice to verify \cref{eq:AMP-engine-eq} where $\mhat_1$ will be an arbitrary cluster point $\nuhat_\infty$ of the sequence $(\nu_n^{(k)})_{n \geq 1}$ and $\mhat_2$ will be the measure $\mu$. 

The route towards establishing \cref{eq:AMP-engine-eq} begins at Jensen's formula.  Letting $C(0,1)$ denote the unit circle, one may apply Jensen's formula to the ratio $S_n := \frac{P_n^{(k)}}{k! P_n}$ to obtain \begin{equation}\label{eq:jensen-intro}
    \sum_{\rho} \log_- |\psi(\rho)| -  \sum_{\zeta} \log_-|\psi(\zeta)| \leq \max_{z \in \psi^{-1}(C(0,1))} \log | S_n(z) |  - \log |S_n(\psi^{-1}(0))|
\end{equation}
where $\{\rho\}$ enumerates the roots of $P^{(k)}$ and $\{\zeta\}$ enumerates the roots of $P$ (see \cref{fact:Jensen}).  Our task then is to control the right-hand side.  The term with the maximum is fairly straightforward to control almost-surely (\cref{lem:upper-bound-potential}), and it is the term $\log |S_n(\psi^{-1}(0)|$ that is more challenging.  In particular, if we set $a = \psi^{-1}(0)$ then \begin{equation} \label{eq:Sn-def-intro}
    S_n(a) = \sum_{1 \leq i_1 < \ldots < i_k \leq n} \frac{1}{a - Z_{i_1}}\cdots \frac{1}{a - Z_{i_k}}\,.
\end{equation}

The case of $k = 1$, this is precisely a sum of i.i.d.\ random variables and so we depending on the distribution of $Z$ and the choice of $a$ we may have that $\P(S_n = 0) = \Theta(n^{-1/2})$.  Since we seek almost sure statements, this is too large to apply Borel-Cantelli.  To get around this issue,  Angst-Malicet-Poly look instead at \emph{triples} of M\"obius transformations.  For most such triples $(\psi_1,\psi_2,\psi_3)$ one has that the vector $(\psi_1^{-1}(0),\psi_2^{-1}(0),\psi_{3}^{-1}(0))$ consists of three distinct complex numbers, say $(a,b,c)$.  The vector $(S_n(a),S_n(b),S_n(c))$ now behaves like a sum of three-dimensional random variables.  In particular, a sufficiently general version of Erd\H{o}s's solution to the Littlewood-Offord problem shows that the probability all coordinates of $(S_n(a),S_n(b),S_n(c))$ are small simultaneously decays like $O(n^{-3/2})$, which is now summable.  An application of Borel-Cantelli will allow one to deduce that almost-surely for generic triples of M\"obius transformations and all large enough $n$ we have at least one of $(S_n(\psi_1^{-1}(0)),S_n(\psi_2^{-1}(0)),S_n(\psi_3^{-1}(0)))$ is at least, say, $1$ in modulus.  Working with a given cluster point $\nuhat_\infty$ of $(\nu_n^{(k)})_{n \geq 1}$ and applying \cref{eq:jensen-intro} together with an application of the law of large numbers to handle the sum over $\{\zeta\}$ will prove \cref{eq:AMP-engine-eq}. 

The main challenge in adapting this approach to fixed $k\in\bN$ is to handle the anti-concentration estimate.  In particular, for fixed $k \geq 2$, handling the quantity  $\P(S_n(a) = 0)$ is a non-linear anti-concentration problem, and major open problems remain in this arena.  As an example, one expects that for each $k \geq 2$ one has $\P(S_n(a) = 0) = O(n^{-1/2})$, but this is only known up to subpolynomial factors \cite{meka2016anti}.  Furthermore, we need to consider vectors of such quantities.  Roughly, for each fixed $k$, we need to take $L$ large enough so that for distinct complex numbers $(z_1,\ldots,z_L)$ we have $$ \sum_{n \geq 1} \P(|S_n(z_1)| \leq 1, \ldots, |S_n(z_L)| \leq 1) < \infty\,. $$

To handle this quantity, we use a \emph{decoupling} approach for anti-concentration.  This was introduced by Costello-Tao-Vu \cite{CTV} in their study of random symmetric matrices and anti-concentration of quadratic forms (see also the survey \cite{nguyen-vu-survey}).  The intuition here is to tackle multilinear anti-concentration problems by comparing them to linear anti-concentration problems, at the cost of decreasing the rate of decay of the resulting bounds.  For us, we need to apply a decoupling lemma to the vector $(S_n(z_1),\ldots,S_n(z_L))$ and handle all coordinates simultaneously in order to obtain a high-dimensional but \emph{linear} anti-concentration problem.  Our main new contribution is the following Lemma: 

\begin{lemma}\label{lem:small-ball-n2}
    Suppose $\mu$ does not have finite support and let $k \in \bN$.  Then for $L  = 2^{k +2}k$ and all pairwise distinct complex numbers $(z_1,\ldots,z_L)$ we have $$\P(|S_n(z_j)| \leq 1 \text{ for all }j \in [L]) \leq \frac{C}{n^2}$$
    where $C > 0$ depends on $k$ and $\mu$.
\end{lemma}

We note that increasing $L$ yields an increase in the exponent on $n$ on the right-hand side; we only need the right-hand side to be summable in $n$.   Further, it is plausible that the right-hand side of \cref{lem:small-ball-n2} should be of the order $n^{-L/2}$; however, even in the case of $k = 2$ and $L = 1$ this is a non-trivial instance of a significant open problem known as the quadratic Littlewood-Offord problem (see \cite{costello-quadratic,CTV,kwan-sauermann,meka2016anti}).  Since we only need summability, the sub-optimal bounds attained by decoupling will be strong enough provided we take $L$ large as in \cref{lem:small-ball-n2}.

We note that this approach differs fundamentally from the anti-concentration approach of our previous work \cite{michelen2022zeros}. In \cite{michelen2022zeros} we deduced our anti-concentration estimate from a powerful theorem of Meka-Nguyen-Vu \cite{meka2016anti} which in turn is proven by a sophisticated Gaussian comparison argument.  

The decoupling approach and proof of \cref{lem:small-ball-n2} is handled in \cref{sec:anticon}.  We then prove \cref{thm:main} in \cref{sec:proof}, and import the necessary tools and adapt ideas from \cite{angst2023almost}. 

\subsection{Notation}

Throughout, the random variables $(Z_j)_{j\geq 1}$ are defined on the common probability space $(\P,\mathcal{F},\Omega)\,.$  The random polynomials we consider are defined by $P_n(z) = (z - Z_1)\cdots (z - z_n)$.  The measure $\mu_n$ is the empirical measure of $P_n$ and the measure $\nu_n^{(k)}$ is the empirical measure of $P_n^{(k)}$.   We will make use of the ratio $S_n = \frac{P_n^{(k)}}{k! P_n}$. 

We set $\lambda_\R$ to be the Lebesgue measure on $\R$ and  $\lambda_\C$ to be the Lebesgue measure on $\C$; we write $\Chat$ for the Riemann sphere.  We denote $\cM = \{\psi(z) = \frac{az + b}{cz + d}\}$ for the set of M\"obius transformations and endow $\cM$ with the measure $\lambda_\cM$ induced by the taking the Lebesgue measure $\lambda_\C^{\otimes 4}$ on the tuples $(a,b,c,d)$ defining the M\"obius transformations.  We denote $\log z=\log_{+}z-\log_{-}z$ where
\[ \log_{-} z = \begin{cases} 
	|\log z|, & 0 \leq z \leq 1, \\
	0, & z \geq 1, 
\end{cases}
\quad\text {and }\quad
\log_{+} z = \begin{cases} 
	0, & 0 \leq z \leq 1, \\
	\log z, & z \geq 1,
\end{cases} \]
where $\log_{-}0=+\infty$.  We write $C(a,r)$ for the circle centered at $a \in \C$ of radius $r$.  For $m \in \bN$ we write $[m] = \{1,2,\ldots,m\}$.

\section{Anticoncentration via decoupling} \label{sec:anticon}

The goal of this section is to prove \cref{lem:small-ball-n2}; we begin with the abstract decoupling lemma of Costello, Tao and Vu \cite{CTV}.  
\begin{lemma}\label{lem:CTV}
    Let $(Y_1,\ldots,Y_r)$ be a collection of random variables taking values in an arbitrary measurable space and let $E = E(Y_1,\ldots,Y_r)$ be an event depending on these variables.  Set $(Y_1',\ldots,Y_r')$ to be an independent copy of the collection of random variables $(Y_1,\ldots,Y_r)$ with the same joint distribution.  Then $$\P(E(Y_1,\ldots,Y_r)) \leq \P\left(\bigwedge_{\alpha \subset [r]} E(Y_1^\alpha,\ldots,Y_r^\alpha) \right)^{1/2^r} $$ 
    where $Y_j^\alpha = Y_j$ if $j \in \alpha$ and $Y_j^\alpha = Y_j'$ if $j \notin \alpha$.
\end{lemma}

The strategy will be to apply this lemma and subsequently take linear combinations of various versions of $S_n(z_j)$ in order to obtain a linear inequality rather than a multi-linear inequality.   We then will need a high-dimensional (linear) anti-concentration statement which is stated in \cite{angst2023almost}.  A random vector $(X_1,\ldots,X_d) \in \C^d$ is \emph{non-degenerate} if there do not exist complex numbers $\alpha_j, \beta$ so that $$\sum_{j  = 1}^d \alpha_j X_j - \beta = 0$$
almost surely.  This non-degeneracy assumption asserts that $(X_1,\ldots,X_d)$ is genuinely $d$-dimensional.  

\begin{prop}[Proposition 2.1 from \cite{angst2023almost}] \label{prop:anticon-AMP}
    Let $(X^n)_{n\geq 1} = (X_1^n,\ldots,X_d^n)_{n\geq 1}$ be a sequence of i.i.d.\ non-degenerate random vectors taking values in $\C^d$ and set $S_n = \sum_{k = 1}^n X^k$.  Then there is a constant $C$ depending on $d,r$ and the law of $X$ so that for all $n$ we have $$\sup_{x \in \C^d} \P\left(\|S_n - x\| \leq r\right) \leq \frac{C}{n^{d/2}}\,.$$ 
\end{prop}

We are now ready to set up the decoupling approach.  Recall that $(Z_j)_{j \geq 1}$ are the i.i.d.\ samples from $\mu$ giving the roots of the polynomials $(P_n)_{n \geq 1}$.  For fixed $n$ and $k$, partition $[n]$ into $k$  disjoint sets $R_1,\ldots,R_k$ with $\lfloor n/k\rfloor \leq |R_j| \leq \lceil n/k \rceil$. For $j \in [k]$ define $Y_j = (Z_i)_{i \in R_j}$.  We now think of the rational function $S_n(z)$ as a function not only of $z$ but also of the quantities $(Y_j)_{j \in [k]}$ and so we write $S_n(z;Y_1,\ldots,Y_k)$ when we want to make this dependence explicit.

Applying \cref{lem:CTV} shows \begin{equation}\label{eq:decouple-application}
\P(|S_n(z_j)| \leq 1 \text{ for all }j \in [L]) \leq \P\left(|S_n(z_j;Y_1^\alpha,\ldots,Y_k^\alpha)| \leq 1 \text{ for all }j \in [L], \alpha \subset [k] \right)^{1/2^k}\,.
\end{equation}
The main use of the decoupling is in the following combinatorial lemma.

\begin{lemma}\label{lem:decouple-consequence}
    Suppose that for all $\alpha \subset [k]$ we have $|S_n(z;Y_1^\alpha,\ldots,Y_k^\alpha)| \leq 1$.  Then $$\prod_{i \in [k]}\left|\sum_{j \in R_i}\left(\frac{1}{z - Z_j} - \frac{1}{z - Z_j'} \right) \right| \leq 2^k\,.$$
\end{lemma}
\begin{proof}
    Define $$h(Y_1,\ldots,Y_k,Y_1',\ldots,Y_k') := \sum_{\alpha \subset [k]}(-1)^{|\alpha|}S_n(z;Y_1^\alpha,\ldots,Y_k^\alpha)\,. $$
    The triangle inequality shows that $|h(Y_1,\ldots,Y_k,Y_1',\ldots,Y_k')| \leq 2^k$\,.  Note that \begin{align*}
        h(Y_1,\ldots,Y_k,Y_1',\ldots,Y_k') = \sum_{\alpha \subset [k]} (-1)^{|\alpha|}\sum_{1 \leq i_1 < i_2 < \ldots < i_k \leq n} \frac{1}{z - Z_{i_1}^\alpha} \cdots \frac{1}{z - Z_{i_k}^\alpha}\,.  
    \end{align*}
    Swapping the sums, we claim that \begin{align*}
        \sum_{\alpha \subset [k]} (-1)^{|\alpha|} \frac{1}{z - Z_{i_1}^\alpha} \cdots \frac{1}{z - Z_{i_k}^\alpha}  = \begin{cases}
            0 & \text{ if } \{i_1,\ldots,i_k\} \cap R_j = \emptyset \text{ for some }R_j; \\
            \prod_{\ell = 1}^k \left(\frac{1}{z - Z_{i_\ell}} - \frac{1}{z - Z_{i_\ell}'}\right) & \text{ otherwise}
        \end{cases}\,.
    \end{align*}
    To see this, assume first that $\{i_1,\ldots,i_k\} \cap R_j = \emptyset \text{ for some }R_j$; then when we sum over $\alpha_j$, the sign changes but the quantity $\frac{1}{z - Z_{i_1}^\alpha} \cdots \frac{1}{z - Z_{i_k}^\alpha}$ does not, thus giving $0$.  Otherwise, we must have that each $R_j$ contains exactly one value from $\{i_1,\ldots,i_k\}$, and thus the sum factors as stated.  This shows $$h(Y_1,\ldots,Y_k,Y_1',\ldots,Y_k') =\prod_{i \in [k]}\sum_{j \in R_i}\left(\frac{1}{z - Z_j} - \frac{1}{z - Z_j'} \right) $$
    as desired.
\end{proof}

We now use \cref{lem:decouple-consequence} to identify a high-dimensional but \emph{linear} anti-concentration event lurking in the right-hand side of \cref{eq:decouple-application}.

\begin{corollary}\label{cor:decouple-consequence}
Suppose $z_1,\ldots,z_L$ satisfy
$|S_n(z_i,Y_1^\alpha,\ldots,Y_k^\alpha)| \leq 1$  for all $i \in [L]$ and $\alpha \subset [k]$.  Then there is some $\ell \in [k]$ and a set $I$ with $|I| \geq L/k$ so that for all $i \in I$ we have $$\left|\sum_{j \in R_\ell}\left(\frac{1}{z_i - Z_j} - \frac{1}{z_i - Z_j'} \right)\right| \leq 2\,. $$
\end{corollary}
\begin{proof}
    Apply \cref{lem:decouple-consequence} to note that for each $i \in [L]$
    we can associate some $\ell_i \in [k]$ for which we have $$\left|\sum_{j \in R_{\ell_i}}\left(\frac{1}{z_i - Z_j} - \frac{1}{z_i - Z_j'} \right) \right| \leq 2\,.$$  Since there are only $k$ choices for each $\ell_i$, the pigeonhole principle shows that at least $L/k$ values of $i$ must have the same value of $\ell_i$.
\end{proof}
With an eye towards applying \cref{prop:anticon-AMP}, we confirm non-degeneracy of the summands appearing in \cref{cor:decouple-consequence}:
\begin{fact}\label{fact:generic}
Suppose $\mu$ does not have  finite support.  Let $L \in \bN$ and set $z_1,z_2,\ldots,z_L$ to be pairwise distinct complex numbers.  Let $Z$ and $Z'$ be independent samples from $\mu$.  Then the vector $$\left(\frac{1}{z_1 - Z} - \frac{1}{z_1 - Z'},\ldots, \frac{1}{z_L - Z} - \frac{1}{z_L - Z'} \right)$$ is non-degenerate.
\end{fact}
\begin{proof}
    The proof is similar to the case appearing in \cite{angst2023almost}. Seeking a contradiction, suppose that this random vector is degenerate, and so suppose there are (possibly complex) numbers $\alpha_j$ and $\beta$ so that $$\sum_{j = 1}^L \alpha_j\left(\frac{1}{z_j - Z} - \frac{1}{z_j - Z'}\right) = \beta\,. $$
    Reveal $Z'$, and set $$\beta' = \beta - \sum_{j = 1}^L \frac{\alpha_j}{z_j - Z'}\,, $$
    which implies that almost-surely in $Z$ we have $$\sum_{j = 1}^L \frac{\alpha_j}{z_j - Z} = \beta'\,.$$
    Clearing denominators, this implies that $Z$ is the zero of a polynomial of degree at most $L$, which contradicts our assumption that $\mu$ does not have finite support.
\end{proof}

\noindent We are now ready to prove \cref{lem:small-ball-n2}.
\begin{proof}[Proof of \cref{lem:small-ball-n2}]
    By \cref{eq:decouple-application} and \cref{cor:decouple-consequence} there is a set $R_\ell$ with $|R_\ell| \geq \lfloor n / k \rfloor$ and a set $I$ with $|I| \geq L/k$ so that \begin{equation}
        \P\left(|S_n(z_j)| \leq 1 \text{ for all }j \in [L] \right) \leq \P\left(\left|\sum_{j \in R_\ell}\left(\frac{1}{z_i - Z_j} - \frac{1}{z_i - Z_j'} \right)\right| \leq 2 \text{ for all }i \in I  \right)^{1/2^k}\,.
    \end{equation}
    We now apply \cref{prop:anticon-AMP}---noting that the non-degeneracy condition is guaranteed by \cref{fact:generic}---to bound 
    \begin{equation}\label{eq:application-small-ball}
        \P\left(\left|\sum_{j \in R_\ell}\left(\frac{1}{z_i - Z_j} - \frac{1}{z_i - Z_j'} \right)\right| \leq 2 \text{ for all }i \in I  \right) \leq \frac{C}{(n/k)^{L/(2k)}}
    \end{equation}
    where $C$ depends on $L$ and $\mu$.  Recalling $L = 2^{k+2}k$ and combining \cref{eq:decouple-application} with \cref{eq:application-small-ball} completes the proof.
\end{proof}

\section{Main Lemmas and Proof of \cref{thm:main}} \label{sec:proof}

Following the Angst-Malicet-Poly strategy  outlined in  \cref{ss:outline}, recall that we have set $P_n(z) = (z - Z_1)\cdots(z - Z_n)$ to be our random polynomial and $$S_n(z) := \frac{P_n^{(k)}(z) }{k! P_n(z)} = \sum_{1 \leq i_1 < i_2 <\ldots < i_k \leq n} \frac{1}{z - Z_{i_1}}\cdots \frac{1}{z - Z_{i_k}}\,.$$
We begin with the following basic consequence of Jensen's formula, proven in \cite[Prop.~2.2]{angst2023almost}\,.

\begin{fact}\label{fact:Jensen}
    Let $S = Q/P$ where $P$ and $Q$ are two polynomials neither of which has $0$ as a root.  Then for each M\"obius transformation $\psi \in \cM$ we have $$\sum_{\rho} \log_-|\psi(\rho)| - \sum_{\zeta} \log_- |\psi(\zeta)| \leq \max_{z \in \psi^{-1}(C(0,1))} \log |S(z)| - \log |S(\psi^{-1}(0))|$$
    where $\{\rho\}$ enumerates the roots of $Q$ up to multiplicity and $\{\zeta\}$ enumerates the roots of $P$.
\end{fact}

Recalling that M\"obius transformations map circles to circles, we aim to handle the max term in \cref{fact:Jensen} first.  Set $C(a,r)$ to be the circle of radius $r$ centered at $a$.  
\begin{lemma}\label{lem:upper-bound-potential}
    There is a set $\cE \subset \Omega \times \C \times \R$ with $\P \otimes \lambda_\C \otimes \lambda_\R (\cE^c) = 0$ so that for all $(\omega,a,r) \in \cE$ we have $$\max_{z \in C(a,r)} \log_+ |S_n(z)| = O(\log n) \,.$$ 
\end{lemma}
\begin{proof}
First consider the case $a=0$.  Note that \begin{align*}
    \int_{\C}\int_{\R} \frac{1}{|r - |z||^{1/2}} \one_{|r - |z|| \leq 1} \,d\lambda_{\R}(r) \,d\mu(z) = \int_{\C} 4 \,d\mu(z) = 4 < \infty
\end{align*}
implying that for $\lambda_\R$-almost-every $r > 0$ we have \begin{equation}\label{eq:r-expectation}
    \E\left[\frac{1}{|r - |Z_1||^{1/2}} \right] \leq 1 + \E\left[\frac{1}{|r - |Z_1||^{1/2}}\one_{|r - Z| \leq 1} \right] < \infty\,.
\end{equation} 
Bound  \begin{equation} \label{eq:bound-log-deriv}
\max_{z \in C(0,r)} |S_n(z)| \leq \left(\sum_{j = 1}^n \frac{1}{|r - |Z_1||} \right)^k \leq n^k \left(\sum_{j = 1}^n \frac{1}{|r - |Z_j||^{1/2}}\right)^{2k}
\end{equation}
and note that for $r$ satisfying \cref{eq:r-expectation} the strong law of large numbers implies that $\P$-almost-surely we have \begin{equation} \label{eq:SLLN-application}
    \sum_{j = 1}^n \frac{1}{|r - |Z_j||^{1/2}} = O(n^{1/2})
\end{equation}
where the implicit constant depends on the instance $\omega \in \Omega$ as well as $r \in \R$\,.  Combining \cref{eq:bound-log-deriv} with \cref{eq:SLLN-application} shows that for $\lambda_\R$-almost-every $r > 0$ and $\P$-almost-surely we have 
\begin{equation}
    \max_{z \in C(0,r)} \log_+ |S_n(z)| = O(\log n)\,.
\end{equation}
To show the same for arbitrary $a \in \C$, simply replace $Z$ with the random variable $Z-a$ to reduce to the case of $a = 0$.  This shows that for every $a \in \C$ there are sets $\Omega_a,U_a$ with $\P(\Omega_a^c) = \lambda_\R(U_a^c) = 0$ so that for all $\omega \in \Omega_a, r \in U_a$ we have that the triple $(a,\omega,r)$ satisfies the hypotheses of the Lemma.  An application of the Fubini-Tonelli theorem completes the proof.
\end{proof}

An application of \cref{lem:small-ball-n2} will allow us to handle the remaining term on the right-hand side of \cref{fact:Jensen}.
\begin{lemma}\label{lem:anti-con-application}
    For each $k \in \bN$, set $L = 2^{k + 2}k$.  Then there is a set $\cE \subset \Omega \times \C^L$ with $\P \otimes \lambda_{\C}^{\otimes L}(\cE^c) = 0$ so that for all $(\omega,z_1,\ldots,z_L) \in \cE$ and all $n$ large enough there is at least one $j \in [L]$ so that $|S_n(z_j)| \geq 1$ (with the convention $|S_n(z)| = +\infty$ if $z$ is a pole of $S_n$).
\end{lemma}
\begin{proof}
    For each $L$-tuple $(z_1,\ldots,z_L)$ of pairwise distinct points, \cref{lem:small-ball-n2} and the Borel-Cantelli lemma show that almost surely for sufficiently large $n$ we have $|S_n(z_j)| \geq 1$ for at least one $j \in [L]$.  Since the set of pairwise distinct $L$-tuples has complement of measure $0$ under the measure $\lambda_{\C}^{\otimes L}$, an application of the Fubini-Tonelli theorem completes the proof. 
\end{proof}

Finally, the strong law of large numbers will control the sum over $\{\zeta\}$:

\begin{fact}\label{fact:LLN-application}
    There is a set $\cE \subset \Omega \times \cM$ with $\P \otimes \lambda_\cM(\cE^c) = 0$ so that for all $(\omega,\psi) \in \cE$ we have $$\lim_{n \to \infty} \int_{\C} \log_-|\psi(z)|\,d\mu_n(z) = \int_{\C} \log_-|\psi(z)|\,d\mu(z)\,.$$
\end{fact}
\begin{proof}
     By definition, we have $\int_{\C} \log_{-} |\psi(z)| d\mu_n(z)=\frac{1}{n}\sum\limits_{k=1}^{n}\log_- |\psi(Z_k(z))|\,d\mu(z)$. By the strong law of large numbers, we have that almost surely
   \begin{equation*}
       \lim_{n \to \infty}  \frac{1}{n}\sum\limits_{k=1}^{n}\log_{-} |\psi(Z_k(z))| = \int_{\C} \log_{-}|\psi(z)|\,d\mu(z)\,.  \qedhere
   \end{equation*}
\end{proof}

We are now ready to verify the assumptions of \cref{lem:AMP-sphere-equality} for cluster points of our sequences of random measures.

\begin{lemma}\label{lem:cluster-inequality}
    Suppose $\mu$ does not have finite support.  There is a set $\cE \subset \Omega \times \cM$ with $\P \otimes \lambda_{\cM}(\cE^c) = 0$ so that the following holds.  For every cluster point $\nuhat_\infty$ of the sequence of empirical probability measures $(\nu_n^{(k)})_{n\geq 1}$ we have $$\int_{\Chat} \log_-|\psi(z)| \,d\nuhat_\infty(z) \leq \int_{\C} \log_- |\psi(z)|\,d\mu(z)$$
    for all $(\omega,\psi) \in \cE$\,.
\end{lemma}
\begin{proof}
By combining \cref{lem:upper-bound-potential}, \cref{lem:anti-con-application}, and \cref{fact:LLN-application}  we see that there is a set $\Omega_1$ with $\P(\Omega_1) = 1$ for which the following holds: for $\lambda_{\cM}^{\otimes L}$-almost-every tuple of M\"obius transformations $(\psi_1,\ldots,\psi_L)$ we have \begin{enumerate}
    \item $\max_{z \in \psi_j^{-1}(C(0,1))} \log |S_n(z)| = O(\log n)$ for all $j \in [L]$. \label{it:logn}
    \item For each $n$ large enough, there is some $j \in [L]$ so that $\log |S_n(\psi_j^{-1}(0)| \geq 0$.  \label{it:anti-con}
    \item $\lim_{n \to \infty} \int_{\C} \log_- |\psi_j(z)| \,d\mu_n(z) = \int_{\C} \log_- |\psi_j(z)| \,d\mu(z)$ for all $j \in [L]$. \label{it:LLN} 
\end{enumerate}
We note that for \cref{it:anti-con} we use the fact that for $\lambda_{\cM}^{\otimes L}$-almost-every tuple $(\psi_j)_{j = 1}^L$, the values $(\psi_j^{-1}(0))_{j = 1}^L$ are pairwise distinct.  
Fix an instance $\omega \in \Omega_1$ and a tuple $(\psi_1,\ldots,\psi_L)$ for which the above three items hold.  
Combining Jensen's formula \cref{fact:Jensen} along with \cref{it:logn} and \cref{it:anti-con}, we see that for each $n$ sufficiently large there is some $j \in L$ for which we have 
 \begin{align}\label{eq:combine-items}
     \left(1 - \frac{k}{n}\right) \int_{\C} \log_-|\psi_j(z)| d\nu_n^{(k)}(z) \le \int_{\C} \log_-|\psi_j(z)| \,d\mu_n(z)  +O\left(\frac{\log n}{n}\right)\,.
\end{align}
By \cref{it:LLN} we have that  \begin{equation}\label{eq:mu_nlim}
  \lim_{n\to \infty} \int_{\C} \log_-|\psi_j(z)| \,d\mu_n(z) =\int_{\C} \log_-|\psi_j(z)| \,d\mu(z)\,.
 \end{equation}

For a cluster point $\nuhat_\infty$, there exists a subsequence $(\nu_{n_i})_{i \geq 1}$ converging to $\nuhat_\infty$.  By truncating the $\log_-$ and applying the monotone convergence theorem we see that for each $j \in [L]$ we have \begin{equation}\label{eq:apply-cluster}
    \int_{\Chat} \log_-|\psi_j(z)| \,d\nuhat_\infty(z) \leq \limsup_{i \to \infty} \int_{\C} \log_-|\psi_j(z)| \,d\nu_{n_i}^{(k)}(z)\,.
\end{equation}

Thinning our subsequence further so that \cref{it:anti-con} holds for a single $j \in [L]$ for all $(n_i)_{i \geq 1}$, we combine \cref{eq:combine-items}, \cref{eq:mu_nlim} and \cref{eq:apply-cluster} to obtain \begin{equation}\label{eq:cluster-ineq}
\int_{\Chat} \log_-|\psi_j(z)| \,d\nuhat_\infty(z) \leq \int_{\C} \log_- |\psi_j(z)|\,d\mu(z)\,.\end{equation}
Defining the set $E \subset \cM$ (depending on the instance $\omega \in \Omega_1$) via \begin{equation*}
    E := \left\{ \psi \in \cM : \exists~\text{cluster point }\nuhat_\infty \text{ with } \int_{\Chat} \log_-|\psi_j(z)| \,d\nuhat_\infty(z) > \int_{\C} \log_- |\psi_j(z)|\,d\mu(z)\right\}
\end{equation*}
we see that \cref{eq:cluster-ineq} shows that $\lambda_{\cM}^{\otimes L}(E^L) = 0$ where $E^L$ is the $L$-fold cartesian product of $E$ with itself; this implies that $\lambda_\cM(E) = 0$ as desired.
\end{proof}

\begin{proof}[Proof of \cref{thm:main}]
    First note that if $\mu$ has finite support, then the theorem follows immediately by the strong law of large numbers.  As such, we will assume throughout that $\mu$ does not have finite support.  By \cref{lem:cluster-inequality}, there is a set $\Omega_1$ with $\P(\Omega_1) = 1$ so that for each $\omega \in \Omega_1$, for $\lambda_\cM$-almost-all $\psi \in \cM$ we have $$\int_{\Chat} \log_-|\psi(z)| \,d\nuhat_\infty(z) \leq \int_{\C} \log_- |\psi(z)|\,d\mu(z)\,. $$  By \cref{lem:AMP-sphere-equality}, this implies that each such cluster point satisfies $\nuhat_\infty = \mu$.  This shows that each subsequence of $\{\nu_n^{(k)}\}_{n \geq 1}$ contains a further subsequence that converges to $\mu$, thus showing that for each $\omega \in \Omega_1$ we have $\nu_n$ converges to $\mu$. 
\end{proof}

\section*{Acknowledgments}
\noindent Both authors are supported in part by NSF grant DMS-2137623.

\bibliographystyle{abbrv}
\bibliography{main.bib}

\end{document}